\documentclass[sn-mathphys-num]{sn-jnl}% Math and Physical Sciences Numbered 

\usepackage{graphicx}%
\usepackage{multirow}%
\usepackage{amsmath,amssymb,amsfonts}%
\usepackage{amsthm}%
\usepackage{mathrsfs}%
\usepackage[title]{appendix}%
\usepackage{xcolor}%
\usepackage{textcomp}%
\usepackage{manyfoot}%
\usepackage{booktabs}%
\usepackage{algorithm}%
\usepackage{algorithmicx}%
\usepackage{algpseudocode}%
\usepackage{listings}%
\usepackage{mathtools}
\usepackage[shortlabels]{enumitem}
\usepackage{dsfont}
\numberwithin{equation}{section}

\newcommand*{\abs}[1]{\left\lvert#1\right\rvert}
\newcommand*{\set}[1]{\left\{#1\right\}}
\newcommand{\real}{\mathbb R}
\newcommand{\ex}{\mathbb E}
\newcommand{\prob}{\mathbb P}
\newcommand{\ND}{\mathcal N}
\newcommand{\borel}{\mathcal B}
\newcommand{\ind}{\mathds 1}
\DeclareMathOperator{\const}{const}
\DeclareMathOperator{\Cov}{Cov}
\DeclareMathOperator{\Var}{Var}
\DeclareMathOperator{\sgn}{sgn}
\newcommand{\BB}{A}
\newcommand{\CC}{B}
\newcommand{\DD}{D}

\allowdisplaybreaks
\newtheorem{theorem}{Theorem}[section]% meant for sectionwise numbers
\newtheorem{proposition}[theorem]{Proposition}% 
\newtheorem{lemma}[theorem]{Lemma}

\theoremstyle{remark}%
\newtheorem{example}[theorem]{Example}%
\newtheorem{remark}[theorem]{Remark}%

\begin{document}

\title[Discretization of integrals w.r.t.\ multifractional Brownian motions]{Discretization of integrals driven by~multifractional Brownian motions with~discontinuous integrands}


\author[1,2]{\fnm{Kostiantyn} \sur{Ralchenko}}\email{kostiantyn.ralchenko@uwasa.fi}

\author[1]{\fnm{Foad} \sur{Shokrollahi}}\email{foad.shokrollahi@uwasa.fi}

\author[1]{\fnm{Tommi} \sur{Sottinen}}\email{tommi.sottinen@uwasa.fi}

\affil[1]{\orgdiv{School of Technology and Innovations}, \orgname{University of Vaasa}, \orgaddress{\street{P.O.~Box~700}, \city{Vaasa}, \postcode{FIN-65101}, \country{Finland}}}

\affil[2]{\orgdiv{Department of Probability, Statistics and Actuarial Mathematics}, \orgname{Taras Shevchenko National University of Kyiv}, \orgaddress{\street{Volodymyrska St.}, \city{Kyiv}, \postcode{01601}, \country{Ukraine}}}


\abstract{We establish the rate of convergence in the $L^1$-norm for equidistant approximations of stochastic integrals with discontinuous integrands driven by multifractional Brownian motion. Our findings extend the known results for the case when the driver is a fractional Brownian motion.}




\maketitle

\section{Introduction}
We consider equidistant approximations of stochastic integrals driven by multifractional Brownian motion with discontinuous integrands. Specifically, we establish the rate of convergence for equidistant approximations of pathwise stochastic integrals:
\[
\int_0^1 \Psi'(X_s) dX_s \approx \sum_{k=1}^n \Psi'\left(X_{t_{k-1}}\right) \left(X_{t_k} - X_{t_{k-1}}\right),
\]
where $t_k = \frac{k}{n}$, $k = 0, 1,\dots, n$. Here, $\Psi$ represents a difference of convex functions, and $X$ denotes a multifractional Brownian motion (see Section~\ref{sec:mfbm} for details). The integral is interpreted as a pathwise Stieltjes integral, following the integration theory for discontinuous integrands developed in \cite{CLV19} using a modification of Z\"ahle's fractional integration theory \cite{Zaehle1, Zaehle2}.

Recently, a similar problem was addressed in \cite{Tommi24} for the case when the driving process $X$ is centered, Gaussian and H\"older continuous  of order $H>\frac12$. Additionally, in \cite{Tommi24},  $X$ satisfies the following conditions: its variance function $V(t)$ is non-decreasing on $[0,1]$, $V(1) = 1$, and its variogram function is represented as
\[
\ex(X_t - X_s )^2 = \sigma^2 \abs{t - s}^{2H} + o\left(\abs{t - s}^{2H}\right), \quad\text{as } |t - s|\to0.
\]
Examples of such processes include fractional, bifractional and sub-fractional Brownian motions, the fractional Ornstein--Uhlenbeck process, and normalized multi-mixed fractional Brownian motion, among others. In \cite{Tommi24}, the exact rate of convergence for approximations of stochastic integrals in the 
$L^1$-distance is found to be proportional to 
$n^{1-2H}$, which corresponds to the known rate in the case of smooth integrands (see \cite{Garino22} and references therein).
Notably, for the case of fractional Brownian motion, this problem was studied earlier in \cite{AV2015}. For other related studies on stochastic integrals with discontinuous integrands, see also \cite{Hinz22, Hinz23, SV16, Yaskov19}.

\looseness=-1 In this paper, we focus on approximating integrals driven by multifractional Brownian motion. This process generalizes fractional Brownian motion by allowing the Hurst index to vary over time. Such a generalization enables the modeling of stochastic processes whose path regularity and ``memory depth'' evolve over time. In this case, the variance function of the process is 
$V(t) = t^{2H_t}$, which is generally non-monotone. Consequently, the direct application of results from \cite{Tommi24} is infeasible, as the proofs there rely on the monotonicity of 
$V(t)$. However, by exploiting the specific form of the variance function, we can address these challenges and establish a rate of convergence proportional to 
$n^{1-2H}$ with $H = \min\{\min_t H_t,\alpha\}$, where $\alpha$ is a H\"older exponent of 
$H_t$. To achieve this, we adapt the general proof scheme from \cite{Tommi24}, but significantly modify and generalize the auxiliary results to accommodate a process with non-monotone variance.

The paper is organized as follows. In Section~\ref{sec:mfbm}, we review various definitions of multifractional Brownian motion and outline its properties necessary for the subsequent sections. Section \ref{sec:main} presents the statement of our main result. The proofs are provided in Section~\ref{sec:proofs}.

\section{Multifractional Brownian motion: Definition and examples}
\label{sec:mfbm}

Let $H\colon[0,1]\to(\frac12,1)$ be a continuous function satisfying the following assumptions:
\begin{enumerate}[({A}1)]
\item\label{A1}
$H_{\min} \coloneqq \min\limits_{t\in[0,1]} H_t > \frac12$
and
$H_{\max} \coloneqq \max\limits_{t\in[0,1]} H_t < 1$.
\item\label{A2}
There exist constants $C>0$ and $\alpha\in(\frac12,1]$ such that for all $t, s \in [0,1]$
\[
\abs{H_t-H_s}\le C\abs{t-s}^\alpha.
\]
\end{enumerate}

There exist several generalizations of fractional Brownian motion to the case where the Hurst index $H$ is varying with time.

\begin{example}[Moving-average multifractional Brownian motion \cite{PeltierLevyVehel}]
\label{ex:moving}
Multifractional Brownian motion was first introduced by Peltier and L\'evy V\'ehel \cite{PeltierLevyVehel}. Their definition is
based on the Mandelbrot--van Ness representation for fractional Brownian motion (see, for example, [6, Chapter 1.3]).
The \emph{moving-average multifractional Brownian motion} is
defined by
\begin{equation}\label{eq:fbm-MvN}
X_t= C_1\left(H_t\right) \int_{-\infty}^t\!\left[(t-s)_+^{H_t-\frac12}-(-s)_+^{H_t-\frac12}\right] dW_s,
\end{equation}
where $W = \{W_t, t \in \real\}$ is a two-sided Wiener process,
$x_+ = \max\{x, 0\}$,
and
\[
C_1(H) = \left(\frac{2H\Gamma\left(\frac32 - H\right)}{\Gamma\left(H + \frac12\right) \Gamma(2 - 2H)}\right)^{1/2}
= \frac{\left(2H \Gamma(2H) \sin(\pi H)\right)^{1/2}}{\Gamma\left(H+\frac12\right)}.
\]
\end{example}

\begin{example}[Multifractional Volterra-type Brownian motion \cite{Boufoussi10,mBm10}]
\label{ex:volterra}
The next definition of a multifractional Brownian motion is based on the integral representation of the fractional Brownian motion through a Brownian motion on a finite interval developed in \cite{NVV99}.
The \emph{multifractional Volterra-type Brownian motion} is the process
\begin{equation}\label{eq:fbm-volterra}
X_t = \int_0^t K_{H_t}(t,s)\,dW_s,
\end{equation}
where $W = \set{W_t, t\ge0}$ is a Wiener process, and $K_H(t,s)$ is the Molchan kernel defined by
\[
K_H(t,s) =
    C_2(H) s^{\frac12-H} \int_s^t(v-s)^{H-\frac32}v^{H-\frac12}\,dv,
\quad H\in(\tfrac12,1),  
\]
with 
$C_2(H) = C_1(H) (H-\frac12)$.
\end{example}

\begin{example}[Harmonizable multifractional Brownian motion \cite{Benassi97,Cohen99}]
\label{ex:harmonizable}
Consider another generalization, introduced in \cite{Benassi97}. Let $W(\cdot)$ be a complex random measure on $\real$ such that
\begin{enumerate}[1)]
\item for all $A,B\in\borel(\real)$,
\[
\ex W(A)\overline{W(B)}=\lambda(A\cap B),
\]
where $\lambda$ is the Lebesgue measure;
\item for an arbitrary sequence
$\set{A_1, A_2,\dots}\subset\borel(\real)$
such that
$A_i\cap A_j=\emptyset$ for all $i \ne j$, we have
\[
W\Biggl(\bigcup_{i\ge1}A_i\Biggr) = \sum_{i\ge1}W(A_i),
\]
(here $\set{W(A_i),i\ge1}$ are centered normal random variables);
\item for all $A\in\borel(\real)$,
    $$
    W(A)=\overline{W(-A)},
    $$
\item for all $\theta\in\real$,
    $$
    \set{e^{i\theta}W(A),A\in\borel(\real)}
    \overset{d}{=}\set{W(A),A\in\borel(\real)}.
    $$
\end{enumerate}
The \emph{harmonizable multifractional Brownian motion} is defined by 
\begin{equation}\label{eq:fbm-harmon}
X_t = C_3(H_t)\int_{\real}\frac{e^{itx}-1}{\abs{x}^{\frac12+H_t}}\,W(dx),
\end{equation}
where
$C_3(H) = (H\Gamma(2H)\sin(\pi H)/\pi)^{1/2}$.
\end{example}

In the sequel, we consider a generalization of the fractional Brownian motion
defined by $X_t=B_t^{H_t}$, $t\in [0,1]$, where
$\set{B^H_t,t\in[0,1],H\in\left(\frac12,1\right)}$
is a family of random variables such that
\begin{enumerate}[(B1)]
\item\label{B1}
for a fixed $H\in\left(\frac12,1\right)$, the process $\set{B^H_t,t\in[0,1]}$ is a fractional Brownian motion with the Hurst parameter $H$;
\item\label{B2}
for all $t\in[0,1]$ and all $H_1,H_2\in [H_{\min},H_{\max}]$,
\begin{equation}\label{eq:bound-fbm}
\ex\left(B_t^{H_1}-B_t^{H_2}\right)^2
\le C (H_1-H_2)^2,
\end{equation}
where
$C$ is a constant that may depend on $H_{\min}$ and $H_{\max}$.
\end{enumerate}

The above conditions are satisfied, for instance, by every one of the generalizations described in Examples~\ref{ex:moving}--\ref{ex:harmonizable}, since conditions \ref{B1} and \ref{B2} hold for representations \eqref{eq:fbm-MvN}--\eqref{eq:fbm-harmon}, see \cite{Cohen99,PeltierLevyVehel,mBm10}.
In particular, the bound \eqref{eq:bound-fbm} for the Mandelbrot--van Ness representation \eqref{eq:fbm-MvN} was established in \cite[proof of Thm.~4]{PeltierLevyVehel}, for the Volterra representation \eqref{eq:fbm-volterra} it was proved in \cite[Eqs.\ (16)--(17)]{mBm10}, and for the harmonizable representation \eqref{eq:fbm-harmon} it can be found in  \cite[proof of Lemma 3.1]{Dozzi18}.

For further reference, we collect necessary properties of the variance and variogram functions of multifractional Brownian motion in the following lemma.
\begin{lemma}\label{l:mfBm}
The multifractional Brownian motion $X = \set{X_t,t\in[0,1]}$ has the following properties.
\begin{enumerate}[(i)]
\item For all $t\in[0,1]$
\[
V(t)\coloneqq\ex X_t^2 = t^{2H_t}.
\]

\item For all $t,s\in[0,1]$
\[
\vartheta(t,s)\coloneqq\ex \left(X_t - X_s\right)^2
\le \abs{t-s}^{2H_{\min}} + C\abs{t-s}^{H_{\min} + \alpha} + C\abs{t-s}^{2\alpha}.
\]
\end{enumerate}
\end{lemma}

\begin{proof}
According to the assumption \ref{B2}, if $H_t=H=\const$, then the process $X_t = B^{H_t}_t$ is a fractional Brownian motion. 
This implies the statement $(i)$ and the following bound
\begin{equation}\label{eq:A1-bound}
\ex \left( B^{H_t}_t - B^{H_t}_s \right)^2 
= \abs{t-s}^{2H_t} \le  \abs{t-s}^{2H_{\min}}.
\end{equation}
Moreover, the assumptions \ref{B2} and \ref{A2} yield
\begin{equation}\label{eq:A2-bound}
\ex \left( B^{H_t}_s - B^{H_s}_s \right)^2 
\le C \left(H_t - H_s\right)^{2} \le  C\abs{t-s}^{2\alpha}.
\end{equation}
Furthermore, by the Cauchy--Schwarz inequality we derive from \eqref{eq:A1-bound} and \eqref{eq:A2-bound} that
\[
\ex\abs{\left( B^{H_t}_t - B^{H_t}_s \right) \left( B^{H_t}_s - B^{H_s}_s \right)}
\le C \abs{t-s}^{H_{\min} + \alpha}
\]
Thus,
\begin{align*}
\ex \left(X_t - X_s\right)^2
& = \ex \left( B^{H_t}_t - B^{H_t}_s \right)^2  
+ \ex \left( B^{H_t}_s - B^{H_s}_s \right)^2
\\
&\quad+ 2 \ex\left[ \left( B^{H_t}_t - B^{H_t}_s \right) \left( B^{H_t}_s - B^{H_s}_s \right)\right]
\\
&\le \abs{t-s}^{2H_{\min}} + C\abs{t-s}^{2\alpha} + C \abs{t-s}^{H_{\min} + \alpha},
\end{align*}
and the claim $(ii)$ is proved.
\end{proof}

\begin{remark}
For a convex function $\Psi$, let $\Psi'$ denote its one sided derivative.
In condition \ref{A1} we assume that the function $H_t$ is bounded away from one. This guarantees that
\[
\int_0^1 \frac{1}{\sqrt{V(s)}}\,ds \le \int_0^1 s^{-H_{\max}}\,ds < \infty.
\]
Then by \cite{CLV19} 
$\int_0^1 \Psi'(X_s) dX_s$
exists as a pathwise Riemann--Stieltjes integral; moreover, it satisfies the following chain rule:
\begin{equation}\label{eq:chain}
\int_0^1 \Psi'(X_s) dX_s = \Psi\left(X_1\right) - \Psi\left(X_0\right).
\end{equation}
\end{remark}

\section{Main result}
\label{sec:main}

Let $t_k = \frac{k}{n}$, $k = 0,1,\dots,n$, be an equidistant partition of the interval $[0,1]$.
Throughout the article we use the notation
\begin{equation}\label{eq:phi}
\varphi(a) \coloneqq \ex\left[Y \ind_{Y>a}\right]
= \frac{1}{\sqrt{2\pi}} e^{-\frac{a^2}{2}},
\quad a\in \real.
\end{equation}
In what follows let $C$ denote a generic constant that may change its value from one occurrence to another.

The following theorem is the main result of the paper.

\begin{theorem}\label{th:main}
Let $X_t = B^{H_t}_t$ be a multifractional Brownian motion with the Hurst function $H_t$ satisfying \ref{A1}--\ref{A2}.
Let $\Psi$ be a convex function with the left-sided derivative $\Psi'$, and let $\mu$ denote the measure associated with the second derivative of $\Psi$ such that 
$\int_\real \varphi(a) \mu(da) < \infty$.
Then for any $\widetilde H \in (\frac12, H_{\min}] \cap (\frac12, \alpha)$,
\begin{multline}\label{eq:main}
\ex \abs{\int_0^1 \Psi'(X_s) dX_s - \sum_{k=1}^n \Psi'\left(X_{t_{k-1}}\right) \left(X_{t_k} - X_{t_{k-1}}\right)}
\\*
\le \int_\real \int_0^1 s^{-H_s} \varphi\left(\frac{a}{s^{H_s}}\right) ds\,\mu(da) \left(\frac1n\right)^{2\widetilde H - 1}
+ \int_\real R_n(a) \mu(da),
\end{multline}
where the remainder satisfies
\begin{equation}\label{eq:remain}
 \int_\real R_n(a) \mu(da) \le C  n^{-\min\{2\widetilde H - H_{\max}, H_{\min}+\alpha - 1, 2\alpha-1\}}.
\end{equation}
\end{theorem}

\begin{remark}
Assumption  $\widetilde H \in (\frac12, H_{\min}] \cap (\frac12, \alpha)$ guarantees that that the remainder is negligible compared to the first term in \eqref{eq:main}.
Indeed, we have 
\[
2\widetilde H - H_{\max} > 2 \widetilde H - 1,
\quad
 H_{\min}+\alpha - 1 > 2 \widetilde H - 1,
\quad\text{and}\quad
2\alpha - 1 > 2 \widetilde H - 1.
\]
Hence,
\[
\frac{\int_\real R_n(a) \mu(da)}{n^{1 - 2 \widetilde H}} \to 0,
\quad \text{as } n\to\infty.
\]
\end{remark}

\begin{remark}
One can formulate the statement of Theorem~\ref{th:main} more precisely by considering the cases $\alpha > H_{\min}$ and $\alpha \in (\frac12, H_{\min}]$ separately.
Evidently, in the case $\alpha > H_{\min}$, \eqref{eq:main} holds with $\widetilde H = H_{\min}$. And in the general case, i.e., $\alpha > \frac12$, one has
\begin{equation}\label{eq:rate-simple}
\ex \abs{\int_0^1 \Psi'(X_s) dX_s - \sum_{k=1}^n \Psi'\left(X_{t_{k-1}}\right) \left(X_{t_k} - X_{t_{k-1}}\right)}
\le C  n^{1 - 2\min\{H_{\min}, \alpha\}}.
\end{equation}
Note that for $\frac12 < \alpha\le H_{\min}$, the leading term in \eqref{eq:main} has the same order $n^{1-2\alpha}$ as the remainder;
so we cannot obtain more precise rate of convergence than \eqref{eq:rate-simple}.
\end{remark}

\begin{remark} 
When the function $H$ is sufficiently smooth and the difference between $H_{\max}$ and $H_{\min}$ is rather small, one can establish a lower bound
in addition to \eqref{eq:main}. 
Namely, under additional assumptions
\begin{equation}\label{eq:additional}
\alpha > H_{\max} \quad \text{and} \quad
3 H_{\max} - 2 H_{\min} < 1,
\end{equation}
the following inequality holds
\begin{multline}\label{eq:lower-bound}
\ex \abs{\int_0^1 \Psi'(X_s) dX_s - \sum_{k=1}^n \Psi'\left(X_{t_{k-1}}\right) \left(X_{t_k} - X_{t_{k-1}}\right)}
\\*
\ge \int_\real \int_0^1 s^{-H_s} \varphi\left(\frac{a}{s^{H_s}}\right) ds\,\mu(da) \left(\frac1n\right)^{2H_{\max} - 1}
+ \int_\real R_n(a) \mu(da),
\end{multline}
where the same remainder that satisfies \eqref{eq:remain}.
Due to assumptions \eqref{eq:additional} the remainder in \eqref{eq:lower-bound} is negligible compared to the leading term.

The proof of the lower bound \eqref{eq:lower-bound} is conducted similarly to that of Theorem~\ref{th:main}, but one uses the inequality
\begin{equation}\label{eq:ge}
\vartheta(t,s)
\ge \abs{t-s}^{2H_{\max}} + g(t,s), \quad
\text{where }
\abs{g(t,s)} \le C\abs{t-s}^{H_{\min} + \alpha},
\end{equation}
instead of Lemma~\ref{l:mfBm} $(ii)$. 
The bound \eqref{eq:ge} is derived similarly to Lemma~\ref{l:mfBm}; the remainder function $g(t,s)$ is the same, namely
\[
\textstyle g(t,s) = \ex ( B^{H_t}_s - B^{H_s}_s )^2
+ 2 \ex[( B^{H_t}_t - B^{H_t}_s )( B^{H_t}_s - B^{H_s}_s)].
\]
\end{remark}

\begin{remark}
In particular, the assumptions \eqref{eq:additional} hold in the case $H_t=H=\const$ (i.e., when $X$ is a fractional Brownian motion). Indeed, in this case one can take $\alpha = 1$, $H_{\min} = H_{\max} = H$, and the bounds \eqref{eq:main} and \eqref{eq:lower-bound} imply that
\begin{multline*}
\ex \abs{\int_0^1 \Psi'(X_s) dX_s - \sum_{k=1}^n \Psi'\left(X_{t_{k-1}}\right) \left(X_{t_k} - X_{t_{k-1}}\right)}
\\*
= \int_\real \int_0^1 s^{-H} \varphi\left(\frac{a}{s^{H}}\right) ds\,\mu(da) \left(\frac1n\right)^{2H - 1}
+ \widetilde R_n(a),
\end{multline*}
with
$\widetilde R_n(a) \le C  n^{-H}$.
This coincides with the result of \cite{Tommi24} for the case of fractional Brownian motion.

Moreover, since in the case $H_t=H=\const$ we have an exact rate of convergence $n^{1 - 2H}$, we see that the result of Theorem~\ref{th:main} cannot be improved substantially.
\end{remark}

\section{Proofs}
\label{sec:proofs}

\subsection{Some auxiliary bounds}
In what follows we will often use the following simple upper bound for small $a$.
\begin{lemma}\label{l:boundedness}
Let $\mu\in\real$. Then for all $\abs{a} \le 1$ and $s>0$
\[
\varphi\left(\frac{a}{s^{\mu}}\right) \le C a^{-2} s^{2\mu} \varphi(a),
\]
where $C = 2e^{-1/2}$ is an absolute constant.
\end{lemma}

\begin{proof}
Denote
$h(x) = x e^{-x}$.
The derivative of $h(x)$ equals
$h'(x) = e^{-x} (1 - x) $, whence
$\max\limits_{x\in\real} h(x) = h(1) = e^{-1}$.
Therefore,
for any $a\in\real$, 
\[
\frac{a^2}{s^{2\mu}}\varphi\left(\frac{a}{s^{\mu}}\right) 
= \frac{2}{\sqrt{2\pi}} h\left(\frac{a^2}{2s^{2\mu}}\right)
\le \frac{2}{e \sqrt{2\pi}}.
\]
Note that $\varphi(a)$ decreases when $\abs{a}$ decreases.
Hence, for $\abs{a} \le 1$ one has
\[
\varphi(a) \ge \varphi (1) = \frac{1}{\sqrt{2\pi e}}.
\]
Combining two obtained inequalities we conclude the proof.
\end{proof}

The next auxiliary result provides an upper bound for an integral for specific power-ex\-po\-nen\-tial integrands. Such integrals often arise in subsequent proofs.
\begin{lemma}\label{l:integral}
Let $\lambda \in \real$ and $\mu\ne 0$. Then for all $\abs{a} \ge 1$,
\[
\int_0^1 s^{\lambda} \,\varphi \left(\frac{a}{s^\mu}\right) ds
\le C a^{-2} \varphi (a).
\]
The constant $C$ may depend on $\lambda$ and $\mu$.
\end{lemma}

\begin{proof}
Denote $a^2 = x$,
\[
F(x) \coloneqq \frac{\int_0^1 s^{\lambda}\varphi\left(\frac{\sqrt{x}}{s^{\mu}}\right) ds}{x^{-1} \varphi (\sqrt{x})}
= \frac{\int_0^1 s^{\lambda} \exp\set{-\frac{x}{2s^{2\mu}}} ds}{x^{-1} \exp\set{-\frac{x}{2}}}.
\]
We need to show that $F$ is bounded on $[1,\infty)$.
By substitution $\frac{x}{2s^{2\mu}} = z$, we have
\[
F(x) = \frac{\frac{1}{2\mu} \left(\frac{x}{2}\right)^{\frac{\lambda+1}{2\mu}}\int_{x/2}^\infty z^{-\frac{\lambda+1}{2\mu}-1} e^{-z}dz}{x^{-1} e^{-x/2}}
= 2^{-\frac{\lambda+1}{2\mu}-1} \frac{\int_{x/2}^\infty z^{-\frac{\lambda+1}{2\mu}-1} e^{-z}dz}{ x^{-\frac{\lambda+1}{2\mu}-1} e^{-\frac{x}{2}}}.
\]

As $x \to \infty$, we get by l'H\^opital's rule
\[
\lim_{x\to\infty} F(x) = 2^{-\frac{\lambda+1}{2\mu}-1} \lim_{x\to\infty}\frac{-\frac12 (\frac{x}{2})^{-\frac{\lambda+1}{2\mu}-1} e^{-\frac{x}{2}}}{-\frac12 x^{-\frac{\lambda+1}{2\mu}-1} e^{-\frac{x}{2}} - \left(\frac{\lambda+1}{2\mu}+1\right)x^{-\frac{\lambda+1}{2\mu}-2} e^{-\frac{x}{2}}}
=2^{-\frac{\lambda+1}{2\mu}-1}.
\]
Taking into account the continuity of the function $F(x)$, we derive its boundedness for all $x\ge 1$.
\end{proof}

\subsection{Approximation estimates}

In this section we present upper bounds for various terms appearing in the proof of the main result.

% Upper bound for I_{1,A,n}
\begin{lemma}
\label{l:I1A}
For all $a \in \real$, 
\[
\sum_{k=2}^n \frac{\left(t_k^{H_{t_k}} - t_{k-1}^{H_{t_{k-1}}}\right)^2}{t_{k-1}^{H_{t_{k-1}}}}\, \varphi\left(\frac{a}{t_{k-1}^{H_{t_{k-1}}}}\right) 
\le C \varphi(a) n^{-\min\{H_{\min}, 2\alpha-1\}}.
\]
\end{lemma}

\begin{proof}
Evidently,
$\varphi\left(\frac{a}{t_{k-1}^{H_{t_{k-1}}}}\right) \le \varphi (a)$. Hence, it suffices to estimate the following sum:
\begin{equation}\label{eq:B1+B2}
\begin{split}
\sum_{k=2}^n \frac{\left(t_k^{H_{t_k}} - t_{k-1}^{H_{t_{k-1}}}\right)^2}{t_{k-1}^{H_{t_{k-1}}}}
&\le 2\sum_{k=2}^n \frac{\left(t_k^{H_{t_k}} - t_k^{H_{t_{k-1}}}\right)^2}{t_{k-1}^{H_{t_{k-1}}}}
+2\sum_{k=2}^n \frac{\left(t_k^{H_{t_{k-1}}} - t_{k-1}^{H_{t_{k-1}}}\right)^2}{t_{k-1}^{H_{t_{k-1}}}}
\\
&\eqqcolon 2(\BB_1 + \BB_2).
\end{split}
\end{equation}

First, let us bound $\BB_1$.
Using the mean value theorem and assumption \ref{A2}, we obtain
\begin{equation}\label{eq:tdiff}
\begin{split}
\abs{t_k^{H_{t_k}} - t_k^{H_{t_{k-1}}}}
&\le \abs{t_k^{H_{\min}} \log t_k} \abs{H_{t_k} - H_{t_{k-1}}}
\le C t_k^{H_{\min}} \abs{\log t_k} \abs{t_k -t_{k-1}}^\alpha
\\
&= C n^{-\alpha} t_k^{H_{\min}} \abs{\log t_k}.
\end{split}
\end{equation}
It is well known that for any $\delta>0$ there exists a constant $C = C(\delta) > 0$ such that
$\abs{\log s} \le C s^{-\delta}$ for all $s\in(0,1]$.
Fix any $0<\delta<H_{\min} - \frac12 H_{\max}$ (this is possible because $H_{\min} > \frac12 > \frac12 H_{\max}$). Then
\[
\abs{t_k^{H_{t_k}} - t_k^{H_{t_{k-1}}}}
\le C n^{-\alpha} t_k^{H_{\min}-\delta}.
\]
Therefore,
\begin{align*}
\BB_1 &\le C n^{-2\alpha} \sum_{k=2}^n \frac{t_k^{2H_{\min}-2\delta}}{t_{k-1}^{H_{\max}}}
=  C n^{-2\alpha}\sum_{k=2}^n \left(\frac{t_k}{t_{k-1}}\right)^{2H_{\min}-2\delta} t_{k-1}^{2H_{\min}-H_{\max}-2\delta}
\\
&\le C n^{1-2\alpha}\cdot \frac1n\sum_{k=1}^n t_{k-1}^{2H_{\min}-H_{\max}-2\delta},
\end{align*}
where we have used the bounds
\[
\left(\frac{t_k}{t_{k-1}}\right)^{2H_{\min}-2\delta}
= \left(\frac{k}{k-1}\right)^{2H_{\min}-2\delta}
= \left(1+\frac{1}{k-1}\right)^{2H_{\min}-2\delta}
\le 2^{2H_{\min}-2\delta},
\quad k\ge2.
\]
Further, the term $\frac1n\sum_{k=1}^n t_{k-1}^{2H_{\min}-H_{\max}-2\delta}$ is bounded as a convergent Riemann sum. Indeed,
\[
\frac1n\sum_{k=1}^n t_{k-1}^{2H_{\min}-H_{\max}-2\delta}
\to\int_0^1 s^{2H_{\min}-H_{\max}-2\delta}ds<\infty,
\quad\text{as } n\to\infty.
\]
Hence
\begin{equation}\label{eq:B1-bound}
\BB_1 \le C n^{1-2\alpha}.
\end{equation}

Next, we estimate $\BB_2$.
Note that the numerators can be bounded as follows:
\begin{align*}
t_k^{H_{t_{k-1}}} - t_{k-1}^{H_{t_{k-1}}}
&= H_{t_{k-1}}\int_{t_{k-1}}^{t_k} x^{H_{t_{k-1}}-1}dx
\le H_{t_{k-1}} t_{k-1}^{H_{t_{k-1}}-1}\! \left(t_k - t_{k-1}\right)
\le \frac{t_{k-1}^{H_{t_{k-1}}-1}}{n}.
\end{align*}
Therefore,
\begin{equation}\label{eq:B2-bound}
\BB_2 \le \frac{1}{n^2} \sum_{k=2}^n t_{k-1}^{H_{t_{k-1}}-2}
\le \frac{1}{n^2} \sum_{k=2}^n t_{k-1}^{H_{\min}-2}
= n^{-H_{\min}} \sum_{k=2}^n (k-1)^{H_{\min}-2}
\le C n^{-H_{\min}},
\end{equation}
since the series $\sum_{k=2}^\infty (k-1)^{H_{\min}-2}$ converges.

Combining \eqref{eq:B1+B2}--\eqref{eq:B2-bound}, we conclude the proof.
\end{proof}

Now we are ready to establish two key lemmas concerning approximations.
The following result is a counterpart of \cite[Lemma 4.10]{Tommi24}.
\begin{lemma}
\label{l:I1B}
For all $a \in \real$, 
\[
\abs{\int_0^1 s^{-H_s} \varphi\left(\frac{a}{s^{H_s}}\right) ds
- \frac1n\sum_{k=2}^n t_{k-1}^{-H_{t_{k-1}}} \varphi\left(\frac{a}{t_{k-1}^{H_{t_{k-1}}}}\right)}
\le C\varphi(a) n^{H_{\max}-1}.
\]
\end{lemma}

\begin{proof}
We start by writing
\begin{align}
\MoveEqLeft[0.8]
\abs{\int_0^1 s^{-H_s} \varphi\left(\frac{a}{s^{H_s}}\right) ds
- \frac1n\sum_{k=2}^n t_{k-1}^{-H_{t_{k-1}}} \varphi\left(\frac{a}{t_{k-1}^{H_{t_{k-1}}}}\right)}
\notag\\
&= \abs{ \int_{0}^{t_1}\! s^{-H_s} \varphi\left(\frac{a}{s^{H_s}}\right) ds
+ \sum_{k=2}^n \int_{t_{k-1}}^{t_k} \!\!\left( s^{-H_s} \varphi\left(\frac{a}{s^{H_s}}\right) 
- t_{k-1}^{-H_{t_{k-1}}} \varphi\left(\frac{a}{t_{k-1}^{H_{t_{k-1}}}}\right)\right) ds}
\notag\\
&\le \CC_0 + \CC_1 + \CC_2,
\label{eq:C0+C1+C2}
\end{align}
where
\begin{align*}
\CC_0 &\coloneqq \int_{0}^{t_1} s^{-H_s} \varphi\left(\frac{a}{s^{H_s}}\right) ds,
\\
\CC_1 &\coloneqq \sum_{k=2}^n \int_{t_{k-1}}^{t_k}\abs{s^{-H_s} \varphi\left(\frac{a}{s^{H_s}}\right) - s^{-H_{t_{k-1}}} \varphi\left(\frac{a}{s^{H_{t_{k-1}}}}\right)}ds,
\\
\CC_2 &\coloneqq \sum_{k=2}^n \int_{t_{k-1}}^{t_k}\abs{ s^{-H_{t_{k-1}}} \varphi\left(\frac{a}{s^{H_{t_{k-1}}}}\right) - t_{k-1}^{-H_{t_{k-1}}} \varphi\left(\frac{a}{t_{k-1}^{H_{t_{k-1}}}}\right)}ds.
\end{align*}

The term $\CC_0$ can be bounded as follows:
\begin{equation}\label{eq:C0-bound}
\CC_0 = \int_{0}^{\frac1n} s^{-H_s} \varphi\left(\frac{a}{s^{H_s}}\right) ds 
\le \varphi(a)\int_{0}^{\frac1n} s^{-H_{\max}} ds
= C \varphi(a) n^{H_{\max}-1}.
\end{equation}

Let us consider $\CC_1$.
\begin{align}
\CC_1 &\le \sum_{k=2}^n \int_{t_{k-1}}^{t_k} \varphi\left(\frac{a}{s^{H_s}}\right) \abs{s^{-H_s}  - s^{-H_{t_{k-1}}}}ds
\notag\\
&\quad +  \sum_{k=2}^n \int_{t_{k-1}}^{t_k}s^{-H_{t_{k-1}}}\abs{\varphi\left(\frac{a}{s^{H_s}}\right) - \varphi\left(\frac{a}{s^{H_{t_{k-1}}}}\right)}ds
\notag\\
&\eqqcolon \CC_{11} + \CC_{12}.
\label{eq:C11+C12}
\end{align}
Since $\varphi\left(\frac{a}{s^{H_s}}\right) \le \varphi(a)$, we see that
\[
\CC_{11} \le \varphi(a)\sum_{k=2}^n \int_{t_{k-1}}^{t_k}\abs{s^{-H_s}  - s^{-H_{t_{k-1}}}}ds.
\]
Using the mean value theorem and the assumption \ref{A2}, we can bound the integrand as follows
\begin{align*}
\abs{s^{-H_s}  - s^{-H_{t_{k-1}}}}
&\le s^{-H_{\max}}\abs{\log s} \abs{H_s - H_{t_{k-1}}}
\\
&\le C s^{-H_{\max}}\abs{\log s} \abs{s - t_{k-1}}^\alpha
\le C s^{-H_{\max}}\abs{\log s} n^{-\alpha}.
\end{align*}
Then
\begin{equation}\label{eq:C11-bound}
\CC_{11} \le C\varphi(a)n^{-\alpha} \int_0^1s^{-H_{\max}}\abs{\log s}ds \le C\varphi(a)n^{-\alpha},
\end{equation}
because the function
$s \mapsto s^{-H_{\max}}\abs{\log s}$ is integrable on $[0,1]$.

In order to estimate $\CC_{12}$, 
note that 
\[
\partial_x \varphi\left(\frac{a}{s^x}\right)
= - \frac{a}{s^x} \varphi\left(\frac{a}{s^x}\right) \partial_x\left(\frac{a}{s^x}\right)
=  \frac{a^2}{s^{2x}} \varphi\left(\frac{a}{s^x}\right) \log s.
\]
Therefore,
\begin{equation}\label{eq:C12-deltaphi}
\varphi\left(\frac{a}{s^{H_s}}\right) - \varphi\left(\frac{a}{s^{H_{t_{k-1}}}}\right)
=  a^2 \log s \int_{H_{t_{k-1}}}^{H_s}  s^{-2x}\varphi\left(\frac{a}{s^x}\right) dx,
\end{equation}
whence
\begin{equation}\label{eq:C12}
\CC_{12} \le a^2\sum_{k=2}^n \int_{t_{k-1}}^{t_k}s^{-H_{\max}}\abs{\log s} \abs{ \int_{H_{t_{k-1}}}^{H_s}  s^{-2x}\varphi\left(\frac{a}{s^x}\right) dx}ds.
\end{equation}

Let us consider two cases separately.

(i) Case $\abs{a}\le 1$.
By Lemma~\ref{l:boundedness},
$a^2 s^{-2x}\varphi\left(\frac{a}{s^x}\right) \le C\varphi(a)$.
Using this bound and Assumption \ref{A2}, we get
\[
a^2 \abs{ \int_{H_{t_{k-1}}}^{H_s}  s^{-2x}\varphi\left(\frac{a}{s^x}\right) dx}
\le C\varphi(a) \abs{H_s - H_{t_{k-1}}}
\le C\varphi(a) \abs{s - t_{k-1}}^{\alpha}
\le C\varphi(a) n^{-\alpha}.
\]
We insert this inequality into \eqref{eq:C12} and obtain
\[
\CC_{12} \le C \varphi(a) n^{-\alpha} \int_{0}^{1}s^{-H_{\max}}\abs{\log s} ds
\le C\varphi(a)n^{-\alpha}.
\]

(ii) Case $\abs{a} > 1$.
The inner integral in \eqref{eq:C12} can be estimated as follows:
\begin{align*}
\abs{ \int_{H_{t_{k-1}}}^{H_s}  s^{-2x}\varphi\left(\frac{a}{s^x}\right) dx}
&\le s^{-2H_{\max}}\varphi\left(\frac{a}{s^{H_{\min}}}\right) \abs{H_s - H_{t_{k-1}}}
\\
&\le C n^{-\alpha} s^{-2H_{\max}}\varphi\left(\frac{a}{s^{H_{\min}}}\right).
\end{align*}
by Assumption \ref{A2}.
Then it follows from \eqref{eq:C12} that
\[
\CC_{12} \le C n^{-\alpha} a^2 \int_{0}^{1}s^{-3H_{\max}}\abs{\log s} \varphi\left(\frac{a}{s^{H_{\min}}}\right)  ds,
\]

Choosing arbitrary $\delta>0$ and applying again the bound $\abs{\log s} \le C_\delta s^{-\delta}$ we get
\[
\CC_{12} \le C n^{-\alpha} a^2 \int_{0}^{1}s^{-3H_{\max}-\delta} \varphi\left(\frac{a}{s^{H_{\min}}}\right)  ds
\le C\varphi(\alpha)n^{-\alpha},
\]
where the last inequality follows from Lemma \ref{l:integral}.

Thus, in both cases, $\CC_{12} \le C\varphi(\alpha)n^{-\alpha}$. Combining this result with \eqref{eq:C11+C12} and \eqref{eq:C11-bound}, we see that
\begin{equation}\label{eq:C1-bound}
\CC_1 \le C\varphi(\alpha)n^{-\alpha} \le C\varphi(\alpha)n^{H_{\max}-1}
\end{equation}
(because $-\alpha < - \frac12 < H_{\max}-1$).

Now, let us consider $\CC_2$. We have
\begin{align*}
\CC_2 &\le \sum_{k=2}^n \int_{t_{k-1}}^{t_k}s^{-H_{t_{k-1}}} \abs{\varphi\left(\frac{a}{s^{H_{t_{k-1}}}}\right) - \varphi\left(\frac{a}{t_{k-1}^{H_{t_{k-1}}}}\right)}ds
\\
&\quad +  \sum_{k=2}^n \int_{t_{k-1}}^{t_k} \varphi\left(\frac{a}{t_{k-1}^{H_{t_{k-1}}}}\right) \abs{s^{-H_{t_{k-1}}} - t_{k-1}^{-H_{t_{k-1}}}}ds
\\
&\eqqcolon \CC_{21} + \CC_{22}.
\end{align*}

Let us estimate $\CC_{21}$.
Using the relation $\varphi'(x) = - x\varphi(x)$, we compute
\begin{equation}\label{eq:derivative}
\partial_u \varphi\left(\frac{a}{u^{H_{t_{k-1}}}}\right)
 = H_{t_{k-1}} a^2 u^{-2H_{t_{k-1}} - 1} \varphi\left(\frac{a}{u^{H_{t_{k-1}}}}\right).
\end{equation}
Then 
$\CC_{21}$ can be rewritten as follows:
\begin{equation}\label{eq:C21}
\CC_{21} = a^2\sum_{k=2}^n H_{t_{k-1}} \int_{t_{k-1}}^{t_k}s^{-H_{t_{k-1}}} \int_{t_{k-1}}^s   u^{-2H_{t_{k-1}} - 1} \varphi\left(\frac{a}{u^{H_{t_{k-1}}}}\right)du\, ds.
\end{equation}

Let us consider two cases.

(i) Case $\abs{a} \le 1$.
By Lemma~\ref{l:boundedness}, 
$ a^2 u^{-2H_{t_{k-1}}} \exp\{-\frac{a^2}{2u^{2H_{t_{k-1}}}}\} \le C\varphi(a)$. Hence, \eqref{eq:C21} yields
\begin{align*}
\CC_{21} &\le C\varphi(a) \sum_{k=2}^n \int_{t_{k-1}}^{t_k}s^{-H_{t_{k-1}}} \int_{t_{k-1}}^s   u^{- 1} du\, ds
\\
&\le C\varphi(a) \sum_{k=2}^n \int_{t_{k-1}}^{t_k}s^{-H_{t_{k-1}}} t_{k-1}^{- 1} \left(s - t_{k-1}\right) ds
\\
&\le C\varphi(a)\sum_{k=2}^n t_{k-1}^{-H_{t_{k-1}}- 1} \int_{t_{k-1}}^{t_k}  \left(s - t_{k-1}\right) ds
\le C\varphi(a)\frac{1}{n^2}\sum_{k=2}^n t_{k-1}^{-H_{t_{k-1}}- 1} 
\\
&\le C\varphi(a)\frac{1}{n^2}\sum_{k=2}^n t_{k-1}^{-H_{\max}- 1} 
= C\varphi(a)n^{H_{\max}- 1} \sum_{k=2}^n (k-1)^{-H_{\max}- 1} 
\\
&\le C\varphi(a)n^{H_{\max}- 1},
\end{align*}
because 
$\sum_{k=2}^n (k-1)^{-H_{\max}-1} \le \sum_{k=2}^\infty (k-1)^{-H_{\max}-1}<\infty$.

(ii) Case $\abs{a} > 1$.
Changing the order of integration in the right-hand side of \eqref{eq:C21}, we obtain
\[
\CC_{21} \le a^2\sum_{k=2}^n \int_{t_{k-1}}^{t_k} u^{-2H_{t_{k-1}} - 1} \varphi\left(\frac{a}{u^{H_{t_{k-1}}}}\right) \int_{u}^{t_k}  s^{-H_{t_{k-1}}}ds\, du.
\]
The inner integral can be bounded as follows:
\[
\int_{u}^{t_k}  s^{-H_{t_{k-1}}}ds
\le u^{-H_{t_{k-1}}} \left(t_k - u\right) 
\le \frac1n u^{-H_{t_{k-1}}}.
\]
Then
\begin{align*}
\CC_{21} &\le \frac1n a^2\sum_{k=2}^n \int_{t_{k-1}}^{t_k} u^{-3H_{t_{k-1}} - 1} \varphi\left(\frac{a}{u^{H_{t_{k-1}}}}\right) du
\\
&\le \frac1n a^2\sum_{k=2}^n \int_{t_{k-1}}^{t_k} u^{-3H_{\max} - 1} \varphi\left(\frac{a}{u^{H_{\min}}}\right) du
\\
&\le \frac1n a^2\int_{0}^{1} u^{-3H_{\max} - 1} \varphi\left(\frac{a}{u^{H_{\min}}}\right) du.
\end{align*}
Applying Lemma~\ref{l:integral}, we get
\[
\CC_{21} \le C \varphi(a) n^{-1} \le C \varphi(a) n^{H_{\max}-1},
\]
that is, we have for $\abs{a} > 1$ the same upper bound for $\CC_{21}$ as in the case $\abs{a} \le 1$.

Now let us consider $\CC_{22}$.
Using the evident bound 
\[
\varphi\left(\frac{a}{t_{k-1}^{H_{t_{k-1}}}}\right) \le \varphi(a),
\]
we get
\begin{equation}\label{eq:C22}
\CC_{22} \le C\varphi(a) \sum_{k=2}^n \int_{t_{k-1}}^{t_k} \left( t_{k-1}^{-H_{t_{k-1}}} - s^{-H_{t_{k-1}}}\right) ds.
\end{equation}
By the mean value theorem, we obtain
\[
t_{k-1}^{-H_{t_{k-1}}} - s^{-H_{t_{k-1}}}
\le H_{t_{k-1}}t_{k-1}^{-H_{t_{k-1}}-1}\left(s-t_{k-1}\right)
\le \frac1n t_{k-1}^{-H_{\max}-1} = \frac{n^{H_{\max}}}{(k-1)^{H_{\max}+1}}.
\]
Substituting this bound into \eqref{eq:C22}, we arrive at
\[
\CC_{22} \le C\varphi(a) n^{H_{\max}-1} \sum_{k=2}^n (k-1)^{-H_{\max}-1} 
\le C\varphi(a) n^{H_{\max}-1}.
\]

Combining the above bounds for $\CC_{21}$ and $\CC_{22}$, we conclude that
\begin{equation}\label{eq:C2-bound}
\CC_{2} \le C\varphi(a) n^{H_{\max}-1}.
\end{equation}

Finally, taking into account the representation \eqref{eq:C0+C1+C2} and the inequalities \eqref{eq:C0-bound}, \eqref{eq:C1-bound} and \eqref{eq:C2-bound}, we complete the proof.
\end{proof}

%Upper bound for $I_{2,n} + I_{3,n}$
\begin{lemma}
\label{l:I23}
Let $Y \sim \ND(0,1)$.
Then for all $a \in \real$, 
\begin{multline*}
\abs{\sum_{k=2}^n \left(t_k^{H_{t_k}} \left[\varphi\left(\frac{a}{t_k^{H_{t_k}}}\right) - \varphi\left(\frac{a}{t_{k-1}^{H_{t_{k-1}}}}\right)\right]\right.\right.
\\
- \left.\left. a\left[\prob\left(Y>\frac{a}{t_k^{H_{t_k}}}\right) - \prob\left(Y>\frac{a}{t_{k-1}^{H_{t_{k-1}}}}\right)\right]\right)}
\le C\varphi(a) n^{-\min\{H_{\min}, 2\alpha-1\}}.
\end{multline*}
\end{lemma}

\begin{proof}
We decompose the left-hand side of the desired inequality as follows
\begin{multline}\label{eq:D-decomp}
\abs{\sum_{k=1}^n \left(t_k^{H_{t_k}} \left[\varphi\left(\frac{a}{t_k^{H_{t_k}}}\right) - \varphi\left(\frac{a}{t_{k-1}^{H_{t_{k-1}}}}\right)\right]\right.\right.
\\
-\left.\left. a\left[\prob\left(Y>\frac{a}{t_k^{H_{t_k}}}\right) - \prob\left(Y>\frac{a}{t_{k-1}^{H_{t_{k-1}}}}\right)\right]\right)}
\\
\le \DD_0 + \DD_1 + \DD_2,
\end{multline}
where
\begin{align}
\DD_0 &= \abs{t_1^{H_{t_1}}\varphi\left(\frac{a}{t_1^{H_{t_1}}}\right)
- a\prob\left(Y>\frac{a}{t_1^{H_{t_1}}}\right)},
\notag
\\
\DD_1 &= \abs{\sum_{k=2}^n \left(t_k^{H_{t_k}} \left[\varphi\left(\frac{a}{t_k^{H_{t_k}}}\right) - \varphi\left(\frac{a}{t_{k}^{H_{t_{k-1}}}}\right)\right]
\right.\right.
\notag\\*
&\qquad\qquad- \left.\left. a\left[\prob\left(Y>\frac{a}{t_k^{H_{t_k}}}\right) - \prob\left(Y>\frac{a}{t_{k}^{H_{t_{k-1}}}}\right)\right]\right)},
\label{eq:D1-def}
\\
\DD_2 &= \abs{\sum_{k=2}^n \left(t_k^{H_{t_k}} \left[\varphi\left(\frac{a}{t_k^{H_{t_{k-1}}}}\right) - \varphi\left(\frac{a}{t_{k-1}^{H_{t_{k-1}}}}\right)\right]\right.\right.
\notag\\*
&\qquad\qquad- \left.\left.a\left[\prob\left(Y>\frac{a}{t_k^{H_{t_{k-1}}}}\right) - \prob\left(Y>\frac{a}{t_{k-1}^{H_{t_{k-1}}}}\right)\right]\right)}.
\label{eq:D2-def}
\end{align}
Let us estimate each term separately.
In order to bound $\DD_0$, we observe that
\begin{equation}\label{eq:D0-ineq}
t_1^{H_{t_1}}\varphi\left(\frac{a}{t_1^{H_{t_1}}}\right)
\ge a\prob\left(Y>\frac{a}{t_1^{H_{t_1}}}\right).
\end{equation}
Indeed, denoting $x = \frac{a}{t_1^{H_{t_1}}}$, we get by \eqref{eq:phi}
\[
x \prob(Y>x) = \ex[x\ind_{Y>x}] \le \ex[Y\ind_{Y>x}] = \varphi(x),
\]
whence \eqref{eq:D0-ineq} follows.
Then taking into account \eqref{eq:D0-ineq}, we may write
\begin{equation}\label{eq:D0-bound}
\DD_0 \le t_1^{H_{t_1}}\varphi\left(\frac{a}{t_1^{H_{t_1}}}\right)
\le t_1^{H_{\min}}\varphi(a) = n^{-H_{\min}} \varphi(a).
\end{equation}

Now let us consider the term $\DD_1$.
Similarly to \eqref{eq:C12-deltaphi}, we have
\begin{equation}\label{eq:D1-deltaphi}
\varphi\left(\frac{a}{t_k^{H_{t_k}}}\right) - \varphi\left(\frac{a}{t_{k}^{H_{t_{k-1}}}}\right)
=  a^2 \log t_k \int_{H_{t_{k-1}}}^{H_{t_k}}  t_k^{-2x}\varphi\left(\frac{a}{t_k^x}\right) dx.
\end{equation}
Furthermore,
\[
\partial_x \prob\left(Y>\frac{a}{t_k^x}\right)
= - \varphi\left(\frac{a}{t_k^x}\right) \partial_x \left(\frac{a}{t_k^x}\right)
= \frac{a}{t_k^{x}} \varphi\left(\frac{a}{t_k^x}\right) \log t_k.
\]
Hence,
\begin{equation}\label{eq:D1-deltaprob}
\prob\left(Y>\frac{a}{t_k^{H_{t_k}}}\right) - \prob\left(Y>\frac{a}{t_{k}^{H_{t_{k-1}}}}\right)
= a \log t_k \int_{H_{t_{k-1}}}^{H_{t_k}}  t_k^{-x}\varphi\left(\frac{a}{t_k^x}\right) dx.
\end{equation}

We insert \eqref{eq:D1-deltaphi} and \eqref{eq:D1-deltaprob} into \eqref{eq:D1-def} and obtain
\begin{align*}
\DD_1 &= \abs{\sum_{k=2}^n \left(t_k^{H_{t_k}} a^2 \log t_k \int_{H_{t_{k-1}}}^{H_{t_k}}  t_k^{-2x}\varphi\left(\frac{a}{t_k^x}\right) dx
- a^2 \log t_k \int_{H_{t_{k-1}}}^{H_{t_k}}  t_k^{-x}\varphi\left(\frac{a}{t_k^x}\right) dx\right)}
\\
&\le a^2 \sum_{k=2}^n \abs{\log t_k}  \abs{\int_{H_{t_{k-1}}}^{H_{t_k}}\varphi\left(\frac{a}{t_k^x}\right)t_k^{-2x}\abs{t_k^{H_{t_k}}   -   t_k^{x}} dx}.
\end{align*}
Using the mean value theorem, we get similarly to \eqref{eq:tdiff}
\[
\abs{t_k^{H_{t_k}}   -   t_k^{x}} \le t_k^{H_{\min}} \abs{\log t_k} \abs{H_{t_k} - x}
\le Cn^{-\alpha}t_k^{H_{\min}} \abs{\log t_k}.
\]
Consequently,
\begin{align*}
\DD_1 &\le C n^{-\alpha} a^2 \sum_{k=2}^n t_k^{H_{\min}} \abs{\log t_k}^2 \abs{\int_{H_{t_{k-1}}}^{H_{t_k}}\varphi\left(\frac{a}{t_k^x}\right)t_k^{-2x}dx}
\\
&\le C n^{-\alpha} a^2 \sum_{k=2}^n t_k^{H_{\min}-2\delta} \abs{\int_{H_{t_{k-1}}}^{H_{t_k}}\varphi\left(\frac{a}{t_k^x}\right)t_k^{-2x}dx},
\end{align*}
where we can choose $\delta\in(0,H_{\min}/2)$.

Let us consider two cases.

(i) Case $\abs{a} \le 1$. 
Using the bound 
$a^2 t_k^{-2x}\varphi\left(\frac{a}{t_k^x}\right) \le C\varphi(a)$ (see Lemma~\ref{l:boundedness}) and Assumption \ref{A2}, we obtain
\begin{align*}
a^2 \abs{\int_{H_{t_{k-1}}}^{H_{t_k}}\varphi\left(\frac{a}{t_k^x}\right)t_k^{-2x}dx}
&\le C\varphi(a) \abs{H_{t_k} - H_{t_{k-1}}}
\\
&\le C\varphi(a) \abs{t_k - t_{k-1}}^{\alpha} = C\varphi(a) n^{-\alpha}.
\end{align*}
Hence,
\[
\DD_1 \le C n^{1 - 2\alpha} \left(\frac1n \sum_{k=1}^n t_k^{H_{\min}-2\delta}\right)  \le C n^{1 - 2\alpha},
\]
since
$\frac1n \sum_{k=1}^n t_k^{H_{\min}-2\delta} \to 
\int_0^1 s^{H_{\min}-2\delta} ds$, as $n\to\infty$.

(ii) Case $\abs{a} > 1$. Since
\begin{align*}
\abs{\int_{H_{t_{k-1}}}^{H_{t_k}}\varphi\left(\frac{a}{t_k^x}\right)t_k^{-2x}dx}
&\le \varphi\left(\frac{a}{t_k^{H_{\min}}}\right)t_k^{-2H_{\max}}
\abs{H_{t_k} - H_{t_{k-1}}}
\\
&\le  C n^{-\alpha} \varphi\left(\frac{a}{t_k^{H_{\min}}}\right)t_k^{-2H_{\max}},
\end{align*}
we see that
\begin{align*}
\DD_1 &\le C n^{1-2\alpha} a^2 \sum_{k=2}^n   \varphi\left(\frac{a}{t_k^{H_{\min}}}\right)t_k^{H_{\min}-2H_{\max}-2\delta} \frac1n.
\end{align*}
Note that
\[
\sum_{k=1}^n   \varphi\left(\frac{a}{t_k^{H_{\min}}}\right)t_k^{H_{\min}-2H_{\max}-2\delta} \frac1n
\to \int_0^1\varphi\left(\frac{a}{s^{H_{\min}}}\right)s^{H_{\min}-2H_{\max}-2\delta} ds,
\quad\text{as } n\to\infty,
\]
where the integral is bounded by $C a^{-2}\varphi(a)$ according to Lemma \ref{l:integral}.
Thus, in this case we also have
\begin{equation}\label{eq:D1-bound}
\DD_1 \le C \varphi(a) n^{1 - 2\alpha}.
\end{equation}

Now it remains to estimate $\DD_2$. 
Using \eqref{eq:derivative}, we can write
\begin{equation}\label{eq:D2-deltaphi}
\varphi\left(\frac{a}{t_k^{H_{t_{k-1}}}}\right) - \varphi\left(\frac{a}{t_{k-1}^{H_{t_{k-1}}}}\right)
= H_{t_{k-1}} a^2 \int_{t_{k-1}}^{t_k}s^{- 2H_{t_{k-1}}-1} \varphi\left(\frac{a}{s^{H_{t_{k-1}}}}\right) ds.
\end{equation}
In addition,
\[
\partial_s \prob\left(Y>\frac{a}{s^{H_{t_{k-1}}}}\right)
= \partial_s \int_{\frac{a}{s^{H_{t_{k-1}}}}}^\infty \varphi(v)dv
= H_{t_{k-1}} a s^{- H_{t_{k-1}}-1} \varphi\left(\frac{a}{s^{H_{t_{k-1}}}}\right)
\]
and
\begin{equation}\label{eq:D2-deltaprob}
\prob\left(Y>\frac{a}{t_k^{H_{t_{k-1}}}}\right) - \prob\left(Y>\frac{a}{t_{k-1}^{H_{t_{k-1}}}}\right)
= H_{t_{k-1}} a \int_{t_{k-1}}^{t_k} s^{- H_{t_{k-1}}-1} \varphi\left(\frac{a}{s^{H_{t_{k-1}}}}\right) ds.
\end{equation}
After substitution \eqref{eq:D2-deltaphi} and \eqref{eq:D2-deltaprob} into \eqref{eq:D2-def} we arrive at
\begin{align}
\DD_2 &= a^2\abs{\sum_{k=2}^n H_{t_{k-1}} \int_{t_{k-1}}^{t_k} s^{- 2H_{t_{k-1}}-1} \varphi\left(\frac{a}{s^{H_{t_{k-1}}}}\right)\left(t_k^{H_{t_k}} 
- s^{H_{t_{k-1}}}\right)ds}
\notag\\
&\le a^2\sum_{k=2}^n H_{t_{k-1}} \int_{t_{k-1}}^{t_k} s^{- 2H_{t_{k-1}}-1} \varphi\left(\frac{a}{s^{H_{t_{k-1}}}}\right) \abs{t_k^{H_{t_k}} 
- t_k^{H_{t_{k-1}}}}ds
\notag\\
&\quad + a^2 \sum_{k=2}^n H_{t_{k-1}} \int_{t_{k-1}}^{t_k} s^{- 2H_{t_{k-1}}-1} \varphi\left(\frac{a}{s^{H_{t_{k-1}}}}\right) \abs{t_k^{H_{t_{k-1}}} - s^{H_{t_{k-1}}}}ds
\notag\\
&\eqqcolon \DD_{21} + \DD_{22}.
\label{eq:D21+D22}
\end{align}
In order to bound $\DD_{21}$, we write using \eqref{eq:tdiff}
\[
\abs{t_k^{H_{t_k}} - t_k^{H_{t_{k-1}}}}
\le C t_k^{H_{\min}} n^{-\alpha} \log n
\le C n^{-\alpha} \log n.
\]
Then
\[
\DD_{21} \le C n^{-\alpha} \log n \sum_{k=2}^n a^2\int_{t_{k-1}}^{t_k} s^{- 2H_{t_{k-1}}-1} \varphi\left(\frac{a}{s^{H_{t_{k-1}}}}\right) ds.
\]

(i) Case $\abs{a} \le 1$.
Since by Lemma~\ref{l:boundedness}
\begin{equation}\label{eq:D2-auxbound}
a^2s^{- 2H_{t_{k-1}}} \varphi\left(\frac{a}{s^{H_{t_{k-1}}}}\right) \le C\varphi(a),
\end{equation}
we see that
\begin{align*}
\DD_{21} &\le C \varphi(a) n^{-\alpha} \log n \sum_{k=2}^n  \int_{t_{k-1}}^{t_k} s^{-1} ds
= C \varphi(a) n^{-\alpha} \log n \int_{t_{1}}^{1} s^{-1} ds
\\
&= C \varphi(a) n^{-\alpha} \log^2 n .
\end{align*}

(ii) Case $\abs{a} > 1$.
We have
\begin{align*}
\DD_{21} &\le C n^{-\alpha} \log n \sum_{k=2}^n a^2\int_{t_{k-1}}^{t_k} s^{- 2H_{\max}-1} \varphi\left(\frac{a}{s^{H_{\min}}}\right) ds
\\
&\le C n^{-\alpha} \log n \, a^2 \int_{0}^{1} s^{- 2H_{\max}-1} \varphi\left(\frac{a}{s^{H_{\min}}}\right) ds
\le C \varphi(a) n^{-\alpha} \log n,
\end{align*}
where the last inequality follows from Lemma~\ref{l:integral}.

Thus, in both cases we have the following upper bound
\[
\DD_{21} \le C \varphi(a) n^{-\alpha} \log^2 n.
\]
For any $\delta>0$, $\log n \le C n^\delta$ for some $C=C(\delta)$.
Therefore choosing $\delta <\frac12(1-\alpha)$, we finally get
\begin{equation}\label{eq:D21-bound}
\DD_{21} \le C \varphi(a) n^{-\alpha + 2\delta} 
\le C \varphi(a) n^{1-2\alpha}.
\end{equation}

Let us consider $\DD_{22}$.
By the mean value theorem,
\[
t_k^{H_{t_{k-1}}} - s^{H_{t_{k-1}}} \le H_{t_{k-1}} s^{H_{t_{k-1}}-1} (t_k - s)
\le s^{H_{\min}-1} \frac1n.
\]
Therefore,
\[
\DD_{22} \le \frac{a^2}{n}\sum_{k=2}^n \int_{t_{k-1}}^{t_k} s^{- 2H_{t_{k-1}}+ H_{\min}-2} \varphi\left(\frac{a}{s^{H_{t_{k-1}}}}\right) ds.
\]
Again, let us consider two cases.

(i) Case $\abs{a} \le 1$.
Applying the bound \eqref{eq:D2-auxbound}, we obtain
\begin{align*}
\DD_{22} &\le  C \varphi(a) n^{-1}\sum_{k=2}^n \int_{t_{k-1}}^{t_k} s^{H_{\min}-2} ds
= C \varphi(a) n^{-1} \int_{n^{-1}}^{1} s^{H_{\min}-2} ds
\\
&= C \varphi(a) n^{-1} \, \frac{n^{1-H_{\min}} - 1}{1-H_{\min}}
\le C \varphi(a) n^{-H_{\min}}.
\end{align*}

(ii) Case $\abs{a} > 1$.
\begin{align*}
\DD_{22} &\le \frac{a^2}{n}\sum_{k=2}^n \int_{t_{k-1}}^{t_k} s^{- 2H_{\max}+ H_{\min}-2} \varphi\left(\frac{a}{s^{H_{\min}}}\right) ds
\\
&\le \frac{a^2}{n}\int_{0}^{1} s^{- 2H_{\max}+ H_{\min}-2} \varphi\left(\frac{a}{s^{H_{\min}}}\right) ds 
\le C \varphi(a) n^{-1}
\le C \varphi(a) n^{-H_{\min}}.
\end{align*}
Hence, in both cases
$\DD_{22} \le C \varphi(a) n^{-H_{\min}}$.
Combining this bound with representation \eqref{eq:D21+D22} and inequality \eqref{eq:D21-bound}, we get 
\begin{equation}\label{eq:D2-bound}
\DD_{2} \le C \varphi(a) n^{-\min\{H_{\min}, 2\alpha-1\}}.
\end{equation}

Now the proof follows from \eqref{eq:D-decomp}, \eqref{eq:D0-bound}, \eqref{eq:D1-bound}, and \eqref{eq:D2-bound}.
\end{proof}

\subsection{Proof of Theorem~\ref{th:main}}
The proof of the main result will be done in two steps.
We start by considering the case $\Psi(x) = (x - a)^+$.
Then the general case will be reduced to that case by application of the following lemma, proved in \cite{Tommi24}.

\begin{lemma}[{\cite[Lemma 4.1]{Tommi24}}]
\label{l:lemma4.1}
Let $\Psi$ be convex and $\psi=\Psi'_{-}$ be its left-sided derivative. Then, for any $x, y \in \real$ we have
\begin{align*}
\Psi(x)-\Psi(y)-\psi(y)(x-y) 
&= \int_{\real}[|x-a|-|y-a|-\sgn(y-a)(x-y)] \mu(d a)
\\
&= 2 \int_{\real}\left[(x-a)^{+} - (y-a)^{+} - \ind_{y>a} (x-y)\right] \mu(d a) 
\\
& \geq 0.
\end{align*}
\end{lemma}

\subsubsection{Case \texorpdfstring{$\Psi(x) = (x - a)^+$}{Psi(x) = (x-a)+}}

\begin{proposition}\label{prop:key}
Let $X$ be a multifractional Brownian motion with the Hurst function $H_t$ satisfying \ref{A1}--\ref{A2}.
Let $\widetilde H \in (\frac12, H_{\min}] \cap (\frac12, \alpha)$. Then for any $a \in \real$,
\begin{equation}\label{eq:key-bound}
\begin{split}
\MoveEqLeft
\ex \abs{\int_0^1 \ind_{X_s > a}\, dX_s - \sum_{k=1}^n \ind_{X_{t_{k-1}} > a} \left(X_{t_{k}} - X_{t_{k-1}}\right)} 
\\*
&\le \frac12 \int_0^1 s^{-H_s} \varphi\left(a s^{-H_s}\right) ds
\left(\frac1n\right)^{2\widetilde H-1} + R_n(a),
\end{split}
\end{equation}
where the remainder satisfies
\begin{equation}\label{eq:remainder}
R_n(a) \le C \varphi(a)\, n^{-\min\{2\widetilde H - H_{\max}, H_{\min}+\alpha - 1, 2\alpha - 1\}}.
\end{equation}
\end{proposition}

\begin{proof}
The proof follows the scheme from \cite[Prop.~4.11]{Tommi24}.
By the chain rule \eqref{eq:chain}, we obtain
\begin{align*}
\int_0^1 \ind_{X_s > a} dX_s 
&= (X_1 - a)^+ - (X_0 - a)^+
\\
&= \sum_{k=1}^n \left[\left(X_{t_k} - a\right)^+ - \left(X_{t_{k-1}} - a\right)^+\right].
\end{align*}
Therefore,
\begin{equation}
\label{eq:int-sum}
\begin{split}
\MoveEqLeft
\int_0^1 \ind_{X_s > a}\, dX_s - \sum_{k=1}^n \ind_{X_{t_{k-1}} > a} \left(X_{t_{k}} - X_{t_{k-1}}\right)
\\
&=\sum_{k=1}^n \left[\left(X_{t_k} - a\right)^+ - \left(X_{t_{k-1}} - a\right)^+ - \ind_{X_{t_{k-1}} > a} \left(X_{t_{k}} - X_{t_{k-1}}\right)\right]
\\
&\ge 0,
\end{split}
\end{equation}
where the last inequality follows from Lemma~\ref{l:lemma4.1}.
Further, from
$(x - a)^+ = x\ind_{x>a} - a\ind_{x>a}$, 
we obtain the following representation for one interval increment:
\begin{multline}\label{eq:incr}
\left(X_{t_k} - a\right)^+ - \left(X_{t_{k-1}} - a\right)^+ - \ind_{X_{t_{k-1}} > a} \left(X_{t_{k}} - X_{t_{k-1}}\right)
\\
= X_{t_k} \ind_{X_{t_k}>a} - X_{t_k} \ind_{X_{t_{k-1}}>a} - a \ind_{X_{t_k}>a} + a \ind_{X_{t_{k-1}}>a}.
\end{multline}
Evidently, we have from \eqref{eq:phi} for $k = 1,\dots,n$
\begin{equation}\label{eq:ex1}
\ex\left[X_{t_k} \ind_{X_{t_k}>a}\right] = \sqrt{V(t_{k})}\varphi\left(\frac{a}{\sqrt{V(t_{k})}}\right)
= t_{k}^{H_{t_k}}\varphi\left(\frac{a}{t_{k}^{H_{t_k}}}\right),
\end{equation}
and
\begin{equation}\label{eq:ex2}
\ex \ind_{X_{t_k}>a} = \prob\left(Y > \frac{a}{t_{k}^{H_{t_k}}}\right),
\end{equation}
where $Y \sim \ND(0,1)$.
Note that the relations \eqref{eq:ex1} and \eqref{eq:ex2} remain valid for $k=0$ under the convention
$\varphi(\pm\infty) = 0$,
$\prob(Y > +\infty) = 0$,
$\prob(Y > -\infty) = 1$.

In order to compute 
$\ex[X_{t_k} \ind_{X_{t_{k-1}}>a}]$,
we denote
\[
\gamma_1 = 0, \quad
\gamma_k = \frac{\Cov\left(X_{t_k}, X_{t_{k-1}}\right)}{\Var X_{t_{k-1}}},
\; k = 2,\dots, n,
\]
and use the representation
\[
X_{t_k} = \gamma_k X_{t_{k-1}} + b_k Y_k,
\]
where $Y_k \sim \ND(0,1)$ is independent of $X_{t_{k-1}}$ and $b_k$ is a normalizing constant (that is, $b_k^2 = \Var X_{t_k} - \gamma_k^2 \Var X_{t_{k-1}}$).
Then we get
\begin{equation}\label{eq:ex3}
\ex\left[X_{t_k} \ind_{X_{t_{k-1}}>a}\right]
= \gamma_k \ex\left[X_{t_{k-1}}\ind_{X_{t_{k-1}}>a}\right]
= \gamma_k t_{k-1}^{H_{t_{k-1}}}\varphi\left(\frac{a}{t_{k-1}^{H_{t_{k-1}}}}\right).
\end{equation}
Combining \eqref{eq:incr}--\eqref{eq:ex3} and rearranging terms, we obtain
\begin{align*}
\MoveEqLeft[0]
\ex \left[ \left(X_{t_k} - a\right)^+ - \left(X_{t_{k-1}} - a\right)^+ - \ind_{X_{t_{k-1}} > a} \left(X_{t_{k}} - X_{t_{k-1}}\right) \right]
\\
&= t_{k}^{H_{t_k}}\varphi\left(\frac{a}{t_{k}^{H_{t_k}}}\right)
- \gamma_k t_{k-1}^{H_{t_{k-1}}}\varphi\left(\frac{a}{t_{k-1}^{H_{t_{k-1}}}}\right) 
+ a \prob\left(Y > \frac{a}{t_{k-1}^{H_{t_{k-1}}}}\right)
- a \prob\left(Y > \frac{a}{t_{k}^{H_{t_k}}}\right)
\\
&= \left[t_{k}^{H_{t_k}} - \gamma_k t_{k-1}^{H_{t_{k-1}}}\right]
\varphi\left(\frac{a}{t_{k-1}^{H_{t_{k-1}}}}\right) 
+ t_{k}^{H_{t_k}} \left[ \varphi\left(\frac{a}{t_{k}^{H_{t_k}}}\right)
- \varphi\left(\frac{a}{t_{k-1}^{H_{t_{k-1}}}}\right) \right]
\\
&\quad+ a \prob\left(Y > \frac{a}{t_{k-1}^{H_{t_{k-1}}}}\right)
- a \prob\left(Y > \frac{a}{t_{k}^{H_{t_k}}}\right).
\end{align*}

Therefore
\[
\ex \abs{\int_0^1 \ind_{X_s > a}\, dX_s - \sum_{k=1}^n \ind_{X_{t_{k-1}} > a} \left(X_{t_{k}} - X_{t_{k-1}}\right)} 
= I_{1,n} + I_{2,n} + I_{3,n},
\]
where
\begin{align*}
I_{1,n} &= \sum_{k=2}^n \left[t_k^{H_{t_k}} - \gamma_k t_{k-1}^{H_{t_{k-1}}}\right] \varphi\left(\frac{a}{t_{k-1}^{H_{t_{k-1}}}}\right),
\\
I_{2,n} &= \sum_{k=1}^n t_k^{H_{t_k}} \left[\varphi\left(\frac{a}{t_{k}^{H_{t_{k}}}}\right)-\varphi\left(\frac{a}{t_{k-1}^{H_{t_{k-1}}}}\right)\right],
\\
I_{3,n}  &= \sum_{k=1}^n  \left[a \prob\left(Y>\frac{a}{t_{k-1}^{H_{t_{k-1}}}}\right) - a\prob\left(Y>\frac{a}{t_k^{H_{t_k}}}\right)\right].
\end{align*}

Applying \cite[Lemma 4.4]{Tommi24} we may write
\[
I_{1,n} = I_{1,A,n} + I_{1,B,n},
\]
where
\begin{align*}
I_{1,A,n} &= - \sum_{k=2}^n \frac{\left(t_k^{H_{t_k}} - t_{k-1}^{H_{t_{k-1}}}\right)^2}{2 t_{k-1}^{H_{t_{k-1}}}} \varphi\left(\frac{a}{t_{k-1}^{H_{t_{k-1}}}}\right),
\\
I_{1,B,n} &= \sum_{k=2}^n \frac{\vartheta(t_k,t_{k-1})}{2 t_{k-1}^{H_{t_{k-1}}}} \varphi\left(\frac{a}{t_{k-1}^{H_{t_{k-1}}}}\right).
\end{align*}
By Lemma \ref{l:I1A},
\[
\abs{I_{1,A,n}}
\le C \varphi(a) n^{-\min\{H_{\min}, 2\alpha-1\}}.
\]
Observing that $H_{\min} > 2H_{\min} - H_{\max}\ge 2\widetilde H - H_{\max}$, we obtain
\[
\abs{I_{1,A,n}}
\le C \varphi(a) n^{-\min\{2\widetilde H-H_{\max}, 2\alpha-1\}}.
\]
Let us consider $I_{1,B,n}$.
Applying Lemma~\ref{l:mfBm} we get
\begin{equation}\label{eq:theta-bound}
\vartheta(t_k,t_{k-1})\le n^{-2H_{\min}} + C n^{-H_{\min}- \alpha} + C n^{-2\alpha}
\le n^{-2\widetilde H} + C n^{-\min\{H_{\min}+ \alpha, 2\alpha\}}.
\end{equation}
Moreover, by Lemma \ref{l:I1B}
\[
\frac1n\sum_{k=2}^n t_{k-1}^{-H_{t_{k-1}}} \varphi\left(\frac{a}{t_{k-1}^{H_{t_{k-1}}}}\right)
= \int_0^1 s^{-H_s} \varphi\left(\frac{a}{s^{H_s}}\right) ds
+ R_{2,B,n},
\]
where 
$R_{2,B,n}\le C\varphi(a) n^{H_{\max}-1}$.
Hence,
\begin{align*}
I_{1,B,n} &\le \Bigl( n^{-2\widetilde H} + C n^{-\min\{H_{\min}+ \alpha, 2\alpha\}}\Bigr)\sum_{k=2}^n \frac{1}{2 t_{k-1}^{H_{t_{k-1}}}} \varphi\left(\frac{a}{t_{k-1}^{H_{t_{k-1}}}}\right)
\\
&= \frac12 \Bigl( n^{1-2\widetilde H} + C n^{1-\min\{H_{\min}+ \alpha, 2\alpha\}}\Bigr) \left(\int_0^1 s^{-H_s} \varphi\left(\frac{a}{s^{H_s}}\right) ds
+ R_{2,B,n}\right)
\\
&= \frac12 n^{1-2\widetilde H}\int_0^1 s^{-H_s} \varphi\left(\frac{a}{s^{H_s}}\right) ds
+ R_{2,B,n}' + R_{2,B,n}'',
\end{align*}
where
\begin{align*}
R_{2,B,n}' &= C n^{1-\min\{H_{\min}+ \alpha, 2\alpha\}} \int_0^1 s^{-H_s} \varphi\left(\frac{a}{s^{H_s}}\right) ds\\
&\le C n^{1-\min\{H_{\min}+ \alpha, 2\alpha\}} \varphi(a) \int_0^1 s^{-H_{\max}} ds
\le C \varphi(a) n^{-\min\{H_{\min}+ \alpha-1, 2\alpha-1\}}
\end{align*}
and
\begin{align*}
R_{2,B,n}'' &= \Bigl(n^{1-2\widetilde H} + C n^{1-\min\{H_{\min}+ \alpha, 2\alpha\}}\Bigr) R_{2,B,n}
\le C \varphi(a) n^{ H_{\max} - \min\{2\widetilde H, H_{\min}+ \alpha, 2\alpha\}}
\\
&\le C \varphi(a) n^{- \min\{2\widetilde H - H_{\max}, H_{\min}+ \alpha - 1, 2\alpha - 1\}}.
\end{align*}

According to Lemma~\ref{l:I23},
\[
\abs{I_{2,n}+I_{3,n}} \le C\varphi(a) n^{-\min\{H_{\min}, 2\alpha-1\}}
\le C \varphi(a) n^{-\min\{2 \widetilde H - H_{\max}, 2\alpha-1\}}.
\]
Combining all above estimates we conclude the proof.
\end{proof}

\subsubsection{Proof of Theorem~\ref{th:main}}

Using Lemma~\ref{l:lemma4.1} and \eqref{eq:chain}, we have
\begin{align*}
\MoveEqLeft
\Psi\left(X_{1}\right)-\Psi\left(X_{0}\right)-\sum_{k=1}^{n} \Psi'\left(X_{t_{k-1}}\right)\left(X_{t_{k}}-X_{t_{k-1}}\right) 
\\
& =\sum_{k=1}^{n}\left[\Psi\left(X_{t_{k}}\right)-\Psi\left(X_{t_{k-1}}\right)-\Psi'\left(X_{t_{k-1}}\right)\left(X_{t_{k}}-X_{t_{k-1}}\right)\right] 
\\
& = 2\int_{\real} Z_{n}^{+}(a) \mu(d a)
\end{align*}
where
\begin{align*}
Z_{n}^{+}(a) & =\sum_{k=1}^{n}\left[\left(X_{t_{k}}-a\right)^{+}-\left(X_{t_{k-1}}-a\right)^{+} - \ind_{X_{t_{k-1}}>a}\left(X_{t_{k}}-X_{t_{k-1}}\right)\right]
\\
& =\int_{0}^{1} \ind_{X_{s}>a} d X_{s}-\sum_{k=1}^{n} \ind_{X_{t_{k-1}}>a}\left(X_{t_{k}}-X_{t_{k-1}}\right)
\\
& \geq 0,
\end{align*}
see \eqref{eq:int-sum}.
Taking expectation and using Proposition~\ref{prop:key} to compute $\ex Z_{n}^{+}(a)$, we get
\begin{align*}
\MoveEqLeft
\ex\left|\Psi\left(X_{1}\right) - \Psi\left(X_{0}\right)
- \sum_{k=1}^{n} \Psi'\left(X_{t_{k-1}}\right) \left(X_{t_{k}} - X_{t_{k-1}}\right)\right|
\\
& = 2 \int_{\real} \ex Z_{n}^{+}(a) \mu(d a) \\
& = \int_{\real}\int_0^1 s^{-H_s} \varphi\left(a s^{-H_s}\right) ds
\,\mu(d a) \left(\frac1n\right)^{2H_{\min}-1}
+ 2 \int_{\mathbb{R}} R_{n}(a)\, \mu(d a).
\end{align*}
Here, the remainder $R_{n}(a)$ is defined in Proposition~\ref{prop:key} and satisfies
\[
R_n(a) \le C \varphi(a)\, n^{-\min\{2\widetilde H - H_{\max}, H_{\min}+\alpha - 1, 2\alpha - 1\}}.
\]
which is integrable since $\int_{\mathbb{R}} \varphi(a) \mu(d a)<\infty$ by assumption. Similarly, the leading order term is finite by the fact that
$$
\int_0^1 s^{-H_s} \varphi\left(a s^{-H_s}\right) ds
\le \varphi(a) \int_{0}^{1} s^{-H_{\max}} \,ds 
\le C \varphi(a).
$$
This completes the proof.
\qed

\backmatter

%\bmhead{Supplementary information}
%
%If your article has accompanying supplementary file/s please state so here. 
%
%Authors reporting data from electrophoretic gels and blots should supply the full unprocessed scans for key as part of their Supplementary information. This may be requested by the editorial team/s if it is missing.
%
%Please refer to Journal-level guidance for any specific requirements.

\bmhead{Acknowledgement}
Kostiantyn Ralchenko gratefully acknowledges support from the Research Council of Finland, decision number 359815.


\end{document}